\documentclass{article}
 \usepackage{euscript}
\usepackage{amsmath}
\usepackage{graphicx}
\usepackage{amsthm}
\usepackage{amssymb}
\usepackage{epsfig}
\usepackage[all]{xy}
\usepackage{url}
\theoremstyle{theorem}
\newtheorem{theorem}{Theorem}[section]
\newtheorem{corollary}[theorem]{Corollary}

\newtheorem{proposition}[theorem]{Proposition}

\newtheorem{definition}{Definition}[section]

\newtheorem{remark}{Remark}[section]
\numberwithin{equation}{section}
\title{The pasch configuration and Steiner triple systems}
\author{Masood Aryapoor\footnote{E-mail: aryapoor2002@yahoo.com}}
\date{}
\begin{document}
\maketitle
   
\begin{section}{Introduction}
A \textbf{Steiner triple system} of order $n$ (or $STS(n)$) is a pair $(X,S)$ where $X$, called the foundation, is a finite set with $n$ elements and $S$ is a subset of $\mathcal{P}_3(X)$ (the set of all 3-subsets of $X$) whose elements are called blocks, such that 
for every pair of distinct elements of $X$ there is exactly one block in $S$ which contains both of them. It is known that an $STS(n)$ exists iff $n\equiv 1,3\mod{6} $. Such values of $n$ are called admissible. For a detailed study of them and related objects see \cite{STB}. 

Another useful way to represent Steiner triple systems is to use the notion of Steiner quasi-groups. A \textbf{Steiner quasi-group} structure on $X$ is a binary operation $\star$ on $X$ subject to the following conditions:
for every $a,b\in X$, we have  $a\star a=a$,  $a\star b=b\star a$, and  $a\star(a\star b)=b$. It is known that there is a one-to-one correspondence between Steiner triple systems with foundation $X$ and Steiner quasi-group structures on $X$.
The correspondence is given by sending a Steiner triple system $S$ to the binary operation $\star$ where $a\star b=c$ iff $\{a,b,c\}\in S$ or $a=b=c$. In this paper we do not distinguish between these two concepts and use them interchangeably. 

A \textbf{pasch configuration} (or quadrilateral)  in  $\mathcal{P}_3(X)$ is a set of four blocks 
$$\{a,b,c\}, \{a,d,e\},\{f,b,d\},\{f,c,e\}$$
 such that all elements $a,b,c,d,e,f$ are distinct. An $STS(n)$ which does not contain a pasch configuration is called \textbf{anti-pasch} (or quadrilateral-free). 
It has been proved that for every admissible number $n\neq 7,13$ there exists an anti-pasch $STS(n)$, see \cite{AC}. 

The pasch configuration can be used to define an invariant of Steiner triple systems, namely the number of pasch configurations contained in them. So the class of anti-pasch Steiner triple systems is exactly the class for which this invariant is zero. 
One can in fact use the pasch configuration to define more invariants for Steiner triple systems which is the main focus of this paper.   \\

\end{section}
 
\begin{section}{Two invariants for Steiner triple systems }
One invariant of Steiner triple systems is the number of pasch configurations in them, see \cite{ST} for instance. Anti-pasch Steiner triple systems are those for which this number is zero. 
So we need other invariants to study anti-pasch Steiner triple systems. For example one can look for pasch configurations which have exactly three blocks in a given STS. This idea is elaborated in what follows.  

\begin{subsection}{Invariants and their properties}
Suppose that $S$ is an $STS(n)$. Associated to $S$, the sets $A(S)$  and $B(S)$ are defined as follows.   
The set $A(S)$ consists of all blocks $B$ for which $S\cup \{ B\}$ contains a pasch configuration whose set of  blocks contains $B$.  In terms of the associated quasigroup structure, it is easy to see that, we have 
$$A(S)=\{\{a\star b,b\star c,c\star a\}|a,b,c\in X\text{are distinct and } \{a,b,c\}\notin S\}.$$
 Inspired by this presentation, we set 
$$B(S)=\{\{a\star b,b\star c,c\star a\}| a,b,c\in X\text{ are distinct}\}.$$ 
Note that $\{a\star b,b\star c,c\star a\}=\{a,b,c\}$ iff $\{a,b,c\} \in S$. So we always have $B(S)=A(S)\cup S$.
Set $\alpha(S)$ and $\beta(S)$ to be the cardinals of $A(S)$ and $B(S)$  respectively. To save space, set $\gamma(S)=\beta(S)-\alpha(S)$.
These numbers are interesting invariants of an  $STS(n)$. 

In order to study these invariants, it is useful to define the following maps
$$\phi_S:\mathcal{P}_3(X)\setminus S\to A(S)$$
$$\phi_S(\{a,b,c\})=\{a\star b,b\star c,c\star a\}$$
and
$$\psi_S:\mathcal{P}_3(X)\to B(S)$$
$$\psi_S(\{a,b,c\})=\{a\star b,b\star c,c\star a\}.$$
Note that these maps are well-defined since if $a,b,c$ are distinct then so are $a\star b,b\star c,c\star a$. By the definitions of $A(S)$ and $B(S)$, these maps are onto.  Moreover 
 $\psi_S(\{a,b,c\})=\{a,b,c\}$ iff $\{a,b,c\} \in S$.
\begin{theorem}\label{inequality}
Suppose that $S$ is an $STS(n)$ with $n> 3$. Then we have 
$$\frac{1}{3}{n\choose 2}\leq\alpha(S)\leq\beta(S)\leq \alpha(S)+\frac{1}{3}{n\choose 2}\leq{n\choose 3}.$$
\end{theorem}
\begin{proof}
First we prove that $\frac{1}{3}{n\choose 2}\leq\alpha(S)$. Consider the map 
$\phi_S$.
I claim that the preimage of each block in $A(S)$ has at most $n-3$ elements. Suppose that $\{x,y,z\}\in A(S)$. The set $\phi_S^{-1}(\{x,y,z\})$ consists of blocks $\{a,b,c\}$ such that 
$$x=a\star b,y=b\star c,z=c\star a.$$
If $a$ is fixed then $b$ and $c$ are uniquely determined as $b=y\star a$ and $c=z\star a$. Moreover it is easy to show that $a$ cannot be equal to one of the elements $x,y,x\star y$. So there are at most $n-3$ choices for $a$ and hence $\phi_S^{-1}(\{x,y,z\})$
has at most $n-3$ elements. Since $\mathcal{P}_3(X)\setminus S$ has ${n\choose 3}-\frac{1}{3}{n\choose 2}=\frac{1}{3}{n\choose 2}(n-3) $ elements, we conclude that 
$A(S)$ has at least $\frac{1}{3}{n\choose 2}$ elements, i.e. $\frac{1}{3}{n\choose 2}\leq\alpha(S)$.

\noindent From  $B(S)=A(S)\cup S$, we  immediately obtain  $\alpha(S)\leq\beta(S)\leq \alpha(S)+\frac{1}{3}{n\choose 2}$.

\noindent Finally, since the map $\phi_S$ is onto, we see that $A(S)$ has at most as many elements as $\mathcal{P}_3(X)\setminus S$, i.e. ${n\choose 3}-\frac{1}{3}{n\choose 2}$.
This implies that $\alpha(S)+\frac{1}{3}{n\choose 2}\leq{n\choose 3}.$

\end{proof}

\end{subsection}

\begin{subsection}{Some classes of Steiner triple systems}
 It is interesting to characterize those STS's for which one of the inequalities in Theorem \ref{inequality} becomes an equality. 
This leads to some new and old classes of Steiner triple systems. First we handle those cases which lead to known classes of Steiner triple systems.

If the first inequality happens to be an equality  then we have a familiar class, namely the class of projective geometries. More precisely we have the following.
\begin{proposition}\label{projective}
Suppose that $S$ is an $STS(n)$ with $n>3$. Then the following statements are equivalent: (1) $\beta(S)=\frac{1}{3}{n\choose 2}$,  (2) $\alpha(S)=\frac{1}{3}{n\choose 2}$ and (3) $S$ is isomorphic to a projective geometry $PG(k,2)$ for some $k\geq 2$.
\end{proposition}
\begin{proof}
Clearly (1) implies (2). Now suppose that (2) holds. Consider the map $\phi_S$.
The proof of Theorem \ref{inequality}  reveals that in this case the preimage of each block in $A(S)$ has exactly $n-3$ elements. 
More precisely  
$$\phi_S^{-1}(\{x,y,z\}=\{\{a,x\star a,z\star a\}| a\in X\setminus\{x,y,x\star y\}\}.$$
It can be seen that $a$ cannot be $z$ as well. So $z=x\star y$. This implies that $A(S)=S$ and each block of $S$ is contained in $n-3$ pasch configurations in $S$. Now we count
 the number of pasch configurations in $S$. Consider the set $\Sigma$ of pairs $(B,P)$ where $B\in S$ and $P$ is a pasch configurations in $S$ containing  $B$. Counting the cardinal of $\Sigma$ in two ways shows that 
the number of pasch configurations in $S$ is $\frac{1}{4}({n\choose 3}-\frac{1}{3}{n\choose 2})=\frac{n(n-1)(n-3)}{24}$. This implies that $S$ is isomorphic to a projective geometry $PG(k,2)$ for some $k\geq 2$, see \cite{ST}.

Finally suppose that (3) holds. Then $ \phi(\{a,b,c\})=\{a+b,b+c,c+a\}$ which belongs to $S$ since $(a+b)+(b+c)+(c+a)=0$. This implies that $B(S)\subset S$. So $\beta(S)\leq \frac{1}{3}{n\choose 2}$. We also have 
$\beta(S)\geq \frac{1}{3}{n\choose 2}$ by Theorem \ref{inequality}. Therefore we have $\beta(S)=\frac{1}{3}{n\choose 2}$.

\end{proof}
Next we have the class of anti-pasch Steiner triple systems.  
\begin{proposition}\label{antipasch}
Suppose that $S$ is an $STS(n)$. Then $\beta(S)=\alpha(S)+\frac{1}{3}{n\choose 2}$ iff  $S$ is anti-pasch.
\end{proposition}
\begin{proof}
 Note that if $S$ is anti-pasch, then $B(S)$ is a disjoint union of $A(S)$ and $S$. Hence, $\beta(S)=\alpha(S)+\frac{1}{3}{n\choose 2}$. Conversely, if $\beta(S)=\alpha(S)+\frac{1}{3}{n\choose 2}$, then we must have 
$S\cap A(S)=\emptyset$ since $B(S)=A(S)\cup S$. This implies that $S$ is anti-pasch.
\end{proof}

Now we consider the new classes of Steiner triple systems.
\begin{definition}
A  Steiner triple system  $S$ is said to have enough pasch configurations if $\alpha(S)=\beta(S)$, i.e. $\gamma(S)=0$. 
\end{definition}
In other words a Steiner triple system  $S$  has enough pasch configurations iff every block of $S$ appears in at least one pasch configuration contained in $S$.

Another new class is given below. 
\begin{definition}
A  Steiner triple system $S$ of order $n$  is called  strongly anti-pasch if $\beta(S)={n\choose 3}$.  
\end{definition} 
The following two facts justify the terminology. 
\begin{proposition}\label{strongly}
Suppose that $S$ is an $STS(n)$. Then $\alpha(S)={n\choose 3}-\frac{1}{3}{n\choose 2}$ iff $\beta(S)={n\choose 3}$.
\end{proposition}
\begin{proof}
From $ \beta(S)\leq \alpha(S)+\frac{1}{3}{n\choose 2}\leq{n\choose 3}$, we see that if $\beta(S)={n\choose 3}$  then  $\alpha(S)={n\choose 3}-\frac{1}{3}{n\choose 2}$. Conversely suppose that $\alpha(S)={n\choose 3}-\frac{1}{3}{n\choose 2}$. 
This implies that the map $\phi_S$ is one-to-one. In order to show that $\beta(S)={n\choose 3}$, we need to show that $A(S)\cap S=\emptyset$, i.e. $S$ is anti-pasch. Suppose on the contrary that $S$ contains a pasch, say  the pasch 
$\{a,b,x\},\{a,c,y\},\{b,y,z\},\{c,x,z\}$. Then we have $\phi_S(\{b,c,x\})=\phi_S(\{b,c,y\})=\{a,z,b\star c\}$, contradicting the fact that $\phi_S$ is one-to-one. 
\end{proof}
 \begin{corollary}
Every strongly anti-pasch STS is an anti-pasch STS. 
\end{corollary}
\begin{proof}
If $S$ is strongly anti-pasch then  $\beta(S)={n\choose 3}$ and $\alpha(S)={n\choose 3}-\frac{1}{3}{n\choose 2}$ by Proposition \ref{strongly}. In particular we have $\beta(S)=\alpha(S)+\frac{1}{3}{n\choose 2}$. By proposition \ref{antipasch},
$S$ is anti-pasch.
\end{proof}

\end{subsection}

\begin{subsection}{Invariants of the direct product of Steiner triple systems}
In this part, we compute the invariants of  the direct product of Steiner triple systems in terms of theirs.
First we recall the definition. Suppose that $S$ and $T$ are two Steiner triple systems with foundations $X$ and $Y$. Their direct product $S\times T$ is defined as follows. The foundation of $S\times T$ is $X\times Y$ and
the quasigroup structure is given by $(a,x)\star (b,y)=(a\star b,x\star y)$.
\begin{theorem}\label{direct}
Suppose that $S$ and $T$ are two Steiner triple systems of orders $m$ and $n$ respectively. Then  we have
$$\alpha(S\times T)=6\alpha(S)\beta(T)+6\beta(S)\alpha(T)-6\alpha(S)\alpha(T)+6{n\choose 2}\beta(S)+$$
$$6{m\choose 2}\beta(T)+n\alpha(S)+m\alpha(T)+4{m\choose 2}{n\choose 2}$$
and
$$\beta(S\times T)=6\beta(S)\beta(T)+(6{n\choose 2}+n)\beta(S)+(6{m\choose 2}+m)\beta(T)+4{m\choose 2}{n\choose 2}.$$
We also have $\gamma(S\times T)=6\gamma(S)\gamma(T)+n\gamma(S)+m\gamma(T).$
\end{theorem}
\begin{proof}
Both sets $A(S)$ and $B(S)$ can be partitioned as follows. Suppose that $\{(a,x),(b,y),(c,z)\}\in\mathcal{P}_3(X\times Y)$ is given.  we have the following cases.\\
\textbf{Type 1}: $|\{a,b,c\}|=|\{x,y,z\}|=3$. Then we have  $\{(a,x),(b,y),(c,z)\}\in A(X\times Y)$ iff $\{a,b,c\}\in A(S), \{x,y,z\}\in B(T)$ or $\{a,b,c\}\in B(S), \{x,y,z\}\in A(T)$. We also have  $\{(a,x),(b,y),(c,z)\}\in B(X\times Y)$ iff $\{a,b,c\}\in B(S)$ and  $\{x,y,z\}\in B(T)$.\\
\textbf{Type 2}:$|\{a,b,c\}|=3$ and $|\{x,y,z\}|=2$. We have  $\{(a,x),(b,y),(c,z)\}\in A(X\times Y)$ iff $\{(a,x),(b,y),(c,z)\}\in B(X\times Y)$ iff $\{a,b,c\}\in B(S)$. \\
\textbf{Type 3}:$|\{a,b,c\}|=2$ and $|\{x,y,z\}|=3$. We have  $\{(a,x),(b,y),(c,z)\}\in A(X\times Y)$ iff $\{(a,x),(b,y),(c,z)\}\in B(X\times Y)$ iff $\{x,y,z\}\in B(T)$. \\
\textbf{Type 4}: $|\{a,b,c\}|=3$ and $|\{x,y,z\}|=1$. We have  $\{(a,x),(b,y),(c,z)\}\in A(X\times Y)$ iff $\{a,b,c\}\in A(S)$. We also have $\{(a,x),(b,y),(c,z)\}\in B(X\times Y)$ iff $\{a,b,c\}\in B(S)$.\\
\textbf{Type 5}: $|\{a,b,c\}|=1$ and $|\{x,y,z\}|=3$. We have  $\{(a,x),(b,y),(c,z)\}\in A(X\times Y)$ iff $\{x,y,z\}\in A(T)$. We also have $\{(a,x),(b,y),(c,z)\}\in B(X\times Y)$ iff $\{x,y,z\}\in B(T)$.\\
\textbf{Type 6}: $|\{a,b,c\}|=2$ and $|\{x,y,z\}|=2$. We have  $\{(a,x),(b,y),(c,z)\}\in A(X\times Y)$ and  $\{(a,x),(b,y),(c,z)\}\in B(X\times Y)$ iff $\{x,y,z\}\in B(T)$.\\
It is straightforward to compute how many elements of each type exist. They are given below.\\
 \textbf{Type 1}:    There are $6\alpha(S)\beta(T)+6\beta(S)\alpha(T)-6\alpha(S)\alpha(T)$ elements of this type in $A(X\times Y)$.  There are $6\beta(S)\beta(T)$ elements of this type in $B(X\times Y)$.\\
\textbf{Type 2  and  3}:  There are $6{m\choose 2}\beta(T)+6{n\choose 2}\beta(S)$ elements of these types in $A(X\times Y)$.  We have the same number of elements of these types in $B(X\times Y)$.\\
 \textbf{Type 4 and 5}: There are $n\alpha(S)+m\alpha(T)$ elements of this type in $A(X\times Y)$. There are $n\beta(S)+m\beta(T)$ elements of this type in $B(X\times Y)$. \\
\textbf{Type 6}:  There are $4{m\choose 2}{n\choose 2}$ elements of this type in $A(X\times Y)$. We have the same number of elements of this type in $B(X\times Y)$.\\
 Adding up these numbers gives the formulas.

\end{proof}

The direct products in various classes of Steiner triple systems are discussed in the following proposition.
\begin{corollary}
Suppose that $S$ and $T$ are two STS's   of orders $m$ and $n$ respectively. Then we have the following.\\
(1) $S$ and $T$ are anti-pasch iff $S\times T$ is anti-pasch.\\
(2) $S$ and $T$ are strongly anti-pasch iff $S\times T$ is strongly anti-pasch.\\
(3) $S$ and $T$ have enough pasch configurations  iff $S\times T$ has enough pasch configurations.
\end{corollary}
\begin{proof}
(1) Suppose that $S$ and $T$ are anti-pasch. Then we have $\gamma(S)=\frac{1}{3}{m\choose 2}$ and $\gamma(T)=\frac{1}{3}{n\choose 2}$. So
$$\gamma(S\times T)=6\gamma(S)\gamma(T)+n\gamma(S)+m\gamma(T) =\frac{2}{3}{m\choose 2}{n\choose 2}+\frac{1}{3}{m\choose 2}n+\frac{1}{3}{n\choose 2}m$$
which simplifies to $\frac{1}{3}{mn\choose 2}$. For another proof  see \cite{ST}.\\
Conversely, if $S\times T $ is anti-pasch then 
$$\frac{1}{3}{mn\choose 2}=\gamma(S\times T)=6\gamma(S)\gamma(T)+n\gamma(S)+m\gamma(T) \leq$$
$$\frac{2}{3}{m\choose 2}{n\choose 2}+\frac{1}{3}{m\choose 2}n+\frac{1}{3}{n\choose 2}m=\frac{1}{3}{mn\choose 2}. $$
This implies that $\gamma(S)=\frac{1}{3}{m\choose 2}$ and $\gamma(T)=\frac{1}{3}{n\choose 2}$, i.e. $S$ and $T$ are anti-pasch.\\
(2) Suppose that $S$ and $T$ are strongly anti-pasch. In order to show that $S\times T$ is strongly anti-pasch, we must show that $\beta(S\times T)={mn\choose 3}$. Using the second formula in  Theorem \ref{direct}, we need only prove that 
  $${mn\choose 3}=6{m\choose 3}{n\choose 3}+6{n\choose 2}{m\choose 3}+6{m\choose 2}{n\choose 3}+n{m\choose 3}+m{n\choose 3}+4{m\choose 2}{n\choose 2}$$
which can directly be proved by simplifying the left-hand side. \\
Conversely suppose that $S\times T$ is strongly anti-pasch. Then 
$${mn\choose 3}=\beta(S\times T)=$$
$$6\beta(S)\beta(T)+(6{n\choose 2}+n)\beta(S)+(6{m\choose 2}+m)\beta(T)+4{m\choose 2}{n\choose 2}\leq$$
$$6{m\choose 3}{n\choose 3}+6{n\choose 2}{m\choose 3}+6{m\choose 2}{n\choose 3}+n{m\choose 3}+m{n\choose 3}+4{m\choose 2}{n\choose 2}={mn\choose 3}.$$
Therefore we must have $\beta(S)={m\choose 3}$ and $\beta(T)={n\choose 3}$, i.e. $S$ and $T$ are strongly anti-pasch.\\
(3) This follows directly from the last formula in Theorem \ref{direct}.
\end{proof}

\end{subsection}

\end{section}

\begin{section}{Discussion on the new classes of Steiner triple systems}
The classes of projective geometries and anti-pasch Steiner triple systems have been known. So in this section only the other two classes are discussed. 
The main question is that for which admissible values of $n$ there is an $STS(n)$ in each of the classes. We discuss the two classes separately. \\
 
\begin{subsection}{Steiner triple systems with enough pasch configurations}
The simplest Steiner triple system with enough pasch configurations is the unique STS(7), i.e. the Fano plane. More generally any projective geometry  $PG(k,2)$ (with $k>2$)  has enough pasch configurations. 
Using the product of Steiner triple systems, we see that for any $n$ of the form $(2^{k_1}-1)\cdots (2^{k_m}-1)$ (where $k_1,...,k_m>2$), 
there is an $STS(n)$ which has enough pasch configurations. 

In order to construct more Steiner triple systems with enough pasch configurations, we recall a known construction of Steiner triple systems. Suppose that $q$ is a prime power such that  $q=6t+1$ where $t$ is a natural number. Let $x$ be a primitive element of 
the finite field $F_q$ and set $y=x^{2t}$. Suppose that $0\notin C\subset F_q$  has $t$ elements with the property  that if $x^i\neq x^j\in C$ then $i-j$ is not divisible by $t$. One can associate the following $STS(q)$ to $C$ 
$$S_C=\{c\{1,y,y^2\}+a|c\in C, a\in F_q\}.$$
Concerning pasch configurations in $S_C$, one has the following.
\begin{proposition}
An arbitrary block $B$ of $S_C$ appears in a pasch configuration contained in $S_C$ iff either $2B\in S_C$ or $B/2\in S_C$.
\end{proposition}
To see a proof, see Brouwer.\\
\begin{proposition}
Suppose that $q$ is not divisible by $31$. Then we can choose $C$ such that $S_C$ has enough pasch configurations.
\end{proposition}
\begin{proof}
We need to show that $C$ can be chosen such that for every block $B\in S_C$ we have either $2B\in S_C$ or $B/2\in S_C$. So it is enough to find a choice of  $C$ such that for every $c\in C$ we have either $2c\in C$ or $c/2\in C$. We have $2=x^s$ for some natural number $s\leq 6t$. 
Note that $s$ cannot be divisible by $t$ because otherwise $2^6=1$ in $F_q$, i.e. $q$ is divisible by $31$. An appropriate $C$ can be constructed recursively as follows. Start with the set $C_1=\{1,x^s\}$. This set has the following two properties: (1) if $x^i\neq x^j\in C$ then $i-j$ is not divisible by $t$, and 
(2)  for every $c\in C_1$ we have either $2c\in C_1$ or $c/2\in C_1$. Suppose that non-equal sets $C_1\subset C_2\subset\dots\subset C_i$ are constructed such that they satisfy the above two properties. If for every integer $j$ there is some $x^r\in C_i$ such that $j-r$ is divisible by $t$ then $C_i$ is the required subset and we are done. So let $j$ be an integer such that there does not exist some $x^r\in C_i$ such that $j-r$ is divisible by $t$. If there does not exist some $x^r\in C_i$ such that $j+s-r$ is divisible by $t$, then $C_{i+1}=C_i\cup\{x^j,2x^j\}$ satisfies the two properties and is bigger that $C_i$. So let  
there be some $x^r\in C_i$ such that $j+s-r=mt$ for some integer $m$. Then $C_{i+1}=C_i\cup\{x^{j-mt}\}$ satisfies the two properties and is bigger that $C_i$. This process is clearly finite and after at most $t-1$ steps, we end up with the desired subset $C$ which proves the proposition.

\end{proof}
Combining the above results one obtains the following.
\begin{theorem}
If $n$ is a product of any number of the following natural numbers,  then there is an STS(n) with  enough pasch configurations: (1) primes of the form $6t+1$  (where $t$ is a natural number), 
(2) squares of primes of the $6t-1$  (where $t$ is a natural number),
(3) $2^k-1$ (where $k>2$ is a natural number).
\end{theorem}
\begin{proof}
Just note that the prime number $31=2^6-1$ for which the previous proposition cannot be used, is of the third form hence there is an STS(31) with enough pasch configurations.  
\end{proof}
 
 \end{subsection}

\begin{subsection}{Strongly anti-pasch Steiner triple systems}
The simplest strongly anti-pasch STS is the trivial STS(3). So $STS(3)\times STS(3)$  which is the unique $STS(9)$ is also strongly anti-pasch. More generally   there  is a strongly anti-pasch  $STS(3^m)$ for every natural number  $m$. 
We furthermore have the following.
\begin{proposition}
Every Hall triple system is strongly anti-pasch. 
\end{proposition}
\begin{proof}
Suppose that $S$ is a Hall triple system. In order to show that $S$ is strongly anti-pasch, we need to show that any block not in $S$  belongs to $A(S)$. Suppose that $\{a,b,c\}\notin S$. Since $S$ is a Hall triple system, the subsystem of $S$ generated by elements $a,b,c$ is isomorphic to the 
 unique STS of order 9. Since the unique STS of order 9 is strongly anti-pasch, this implies that adding $\{a,b,c\}$ to $S$ produces a pasch configuration containing $\{a,b,c\}$, i.e. $\{a,b,c\}\in A(S)$.
\end{proof} 

\begin{remark}
(1) Note that in the same way one can prove that if every 3 points of an STS generate a strongly anti-pasch Steiner triple system then the Steiner triple system is itself strongly anti-pasch.\\
(2) It is known that the order of any Hall triple system is a power of 3. So this proposition does not give a new order for which there is a strongly anti-pasch Steiner triple system of that order. 
\end{remark}

 \end{subsection}

\begin{subsection}{Related problems}
There are some questions to be answered and some possible problems to work on. The most important questions are
(1)  for what admissible values of $n$  there exists an $STS(n)$ with enough pasch configurations and 
2) for what admissible values of $n$  there exists a strongly anti-pasch  $STS(n)$? It seems that the answers to 1) and 2) are all admissible values  $n$ except finitely many values, and all perfect powers of $3$, respectively. 

As for the invariants defined in this paper, the main problem is finding the spectrum of these values. The inequalities in \ref{inequality} give a restriction on these invariants. 
However not all values satisfying these inequalities can occur as invariants of some $STS$. For example
it is easy to show that  $\gamma(S)\neq\frac{1}{3}{n\choose 2}-1$  for every Steiner triple system of order $n$.

One can use other configurations, e.g. the Mitre configuration, to define new  invariants for Steiner triple systems as done in this paper. 
It is interesting to see what classes of Steiner triple systems are obtained in this way and how these invariants behave.

\end{subsection}

\textbf{Acknowledgments:} I would like to thank G.B. Khosrovshahi for introducing me to the subject. I would also thank him and E. Ghorbani 
for various exciting and stimulating discussions on the concept of anti-pasch  which greatly inspired me to write this paper.

\end{section}

\end{document}